\documentclass[leqno,12pt]{article}

 \usepackage{amsmath}
\usepackage{amscd}
\usepackage{amsopn}
\usepackage{amsthm}
\usepackage{amsfonts,amssymb}
\usepackage{amsfonts,bbm}
\usepackage{latexsym}

\usepackage{color}

\makeatletter
\@addtoreset{equation}{section}
\makeatother

\newtheorem{lemma}{LEMMA}[section]
\newtheorem{proposition}[lemma]{PROPOSITION}
\newtheorem{corollary}[lemma]{COROLLARY}
\newtheorem{theorem}[lemma]{THEOREM}
\newtheorem{remark}[lemma]{REMARK}

\newcommand{\real}{\mathbbm{R}}
\newcommand{\nat}{\mathbbm{N}}

\renewcommand{\a}{\alpha}

\newcommand{\vp}{\varphi}
\newcommand{\ve}{\varepsilon}

\newcommand{\reald}{{\real^d}}

\newcommand{\und}{\quad\mbox{ and }\quad}
\newcommand{\inv}{^{-1}}

\newcommand{\W}{\mathcal W}

\newcommand{\C}{\mathcal C}  
\newcommand{\E}{{\mathcal E}}
\newcommand{\F}{\mathcal F}

\newcommand{\B}{\mathcal B}

\newcommand{\J}{{\mathcal J}}

\newcommand{\itemframe}%
{\setlength{\parskip}{10pt}\begin{enumerate} \setlength{\topsep}{10pt}%
\setlength{\itemsep}{15pt}\setlength{\parsep}{5pt}}

\newcommand{\schluss}{\end{frame}\end{document}}

\newcommand{\Fk}{\F\!_{co}(G)}
\newcommand{\Fi}{\F\!_{ui}(G)}
\title{Compactness of integral  operators\\ and uniform integrability on measure spaces}
\author{Wolfhard Hansen}
\date{}
\begin{document}

\maketitle

\thispagestyle{empty}

\begin{abstract}
  Let $(E,\mathcal E,\mu)$ be a measure space and $G\colon E\times E\to [0,\infty]$ be measurable.
  Moreover, let $\mathcal F\!_{ui}$   denote the set of all $q\in\E^+$ (measurable numerical functions
  $q\ge 0$ on $E$) such that
  $\{G(x,\cdot)q\colon x\in E\}$ is  uniformly integrable, and let~$\mathcal F\!_{co}$ denote the set of all
  $q\in\mathcal E^+$ such that the mapping $f\mapsto G(fq) :=\int  G(\cdot,y) f(y) q(y)\,d\mu(y)$
  is a compact operator on the space  $\mathcal E_b$ of bounded measurable functions on~$E$
  (equipped with the sup-norm).

  It is shown that $\mathcal F\!_{ui}=\mathcal F\!_{co} $ provided both $\mathcal F\!_{ui}$ and $\mathcal F\!_{co} $
    contain strictly positive functions.
  \end{abstract}

  \section{Introduction, notation and first properties}

  In the paper \cite{bjk} on semilinear perturbation of   fractional Laplacians $-(-\Delta)^{\a/2}$
  on~$\reald$, $d\in\nat$, $0<\a< 2\wedge d$, 
  a~Kato class $\J^\a(D)$ of  measurable  functions $q $ on an open set $D$ in $\reald$
  is defined by uniform   integrability of the functions $G_D(x,\cdot) q$, $x\in D$, with respect to
  Lebesgue measure on $D$,   where $G_D$ denotes the corresponding Green function  on $D$
  (see \cite[Definition~1.23]{bjk}).
  Let $\C_0(D)$, $\B_b(D)$,  respectively, denote the space of all real functions on $D$ 
  which are continuous and vanish at infinity with respect to $D$, are Borel measurable and bounded, respectively.
  
 Suppose that $D$ is regular and let $q\colon D\to [0,\infty]$ be Borel measurable. 
  Then, using the continuity of $G_D$ and Vitali's theorem, it is established
  that, provided $q\in\J^\a(D)$,  the mapping
  \begin{equation*}
    K\colon f\mapsto G_D(fq):=   \int G(_D\cdot,y)f(y) q(y)\,d\mu\in \C_0(D)
    \end{equation*} 
  is a~compact operator on $\B_b(D)$ such that $K(\B_b(D))\subset \C_0(D)$ (see 
  the proof of Theorem 2.4 in \cite{bjk}). In \cite[Proposition~1.31]{bjk} it is stated that, conversely,
  $q\in \J^\a(D)$ if (only) $K1=G_Dq\in\C_0(D)$.

  On the other hand, it is easily seen that, as in the classical case $\a=2$,  compactness of  $K$ on $\B_b(D)$
  implies that $G_Dq$ is continuous, and hence $G_Dq\in\C_0(D)$ due to the regularity of $D$
  and the domination principle for $K$ 
  (see \cite[Corollary 4.5]{H-equi-compact} for bounded $D$ and
  \cite[Corollary 5.2,(b)]{bogdan-hansen-semilinear} for the general case).
  So $q\in\J^\a(D)$  if and only if $K$ is a~compact operator on $\B_b(D)$.
  For some partial statements in the classical case see
  \cite[Corollary to Proposition 3.1, Theorem 3.2]{chung-zhao}.
    
\bigskip

In this paper,  we fix a~ measure space $(E,\E,\mu)$ and a~numerical function $G\ge 0$ on $E\times E$
which is $\E\otimes \E$-measurable.
The purpose of this note is to establish that, even in this most general setting, for every measurable $q\ge 0$ on $E$,
uniform integrability of $\{G(x,\cdot)q\colon x\in E\}$ is equivalent to compactness of
the mapping $f\mapsto \int G(\cdot,y)f(y)q(y)\,d\mu$  
on the space of  bounded measurable functions on $E$  
provided that there are strictly positive functions having these properties.

To be more specific let   $\E^+$, $\E_b$, respectively,
denote  the set of all $\E$-measurable numerical functions $f$ on $E$ such that
$f\ge 0$, $f$ is bounded, respectively. We define 
 \begin{equation}\label{def-Gf}
  Gf(x) :=\int G(x,y)f(y)\,d\mu(y), \qquad f\in \E^+, \ x\in E.
\end{equation}
If $q\in \E^+$ and $Gq$ is bounded,  then  $K_q\colon f\mapsto G(fq)$
obviously is a~bounded operator on $\E_b$, equipped with the sup-norm   $\|\cdot\|_\infty$, such that 
$\|K_q\|= \|Gq\|_\infty$. So we are interested in the sets
\begin{eqnarray*} 
  \Fk&:=&\{q\in\E^+\colon \mbox{$Gq$ is bounded and $K_q$ is a~compact operator on $\E_b$}\},\\
  \Fi&:=&\{q\in\E^+\colon \{G(x,\cdot)q, \ x\in E\} \mbox{ is uniformly  integrable}\}.
\end{eqnarray*}
Our main results are that $\Fi\subset \Fk$ if there exists a~strictly positive function in $\Fk$
(Theorem \ref{k-i}),  and $\Fk\subset \Fi$ if there exists a~strictly positive function in $\Fi$
(Theorem \ref{i-k}).  In Section \ref{suff-ex} we provide general examples (covering the situation in \cite{bjk}),
where  the  assumptions are satisfied. 

Before studying the relation between $\Fi$ and $\Fk$, let us recall that a~subset $\F$ of $\E^+$ is
called \emph{uniformly integrable} if,
 for every $\ve>0$, there exists an integrable $g\in\E^+$ such that $\int_{\{f\ge g\}} f\,d\mu\le \ve$
 for every $f\in\F$.  Further, we note some simple facts,  which are trivial for $\Fi$. 

       \begin{lemma}\label{simple} Let $\F=\Fi$ or $\F=\Fk$. The following holds:
         \begin{itemize}
         \item[\rm (1)]
           $\F$ is a~convex cone and $ G(1_{\{q=\infty\}} q)=0$ for all $q\in \F$.
         \item[\rm (2)]
           If $q\in\F$ and $q'\in\E^+$ with $q'\le q$, then $q'\in \F$.
         \item[\rm (3)]
           If $q\in\E^+$ and, for every $\ve>0$, there are $q'\in\F$ and $q''\in\E^+$ with $q=q'+q''$
           and  $Gq''\le \ve$,  then $q\in\F$.
         \end{itemize}
       \end{lemma}

     Clearly,   (1)  holds for $\Fk$ as well, and to get (2) it suffices to observe that  taking $h:=1_{\{q>0\}} q'/q$
     we have $G(fq')=G(fhq)$ for every $f\in\E_b$. For (3) it suffices to note that $\|K_{q''}\|=\|Gq''\|_\infty$
     and every limit of compact operators on $\E_b$ is compact.

\section{Relation between $\Fk$ and $\Fi$}

\begin{theorem}\label{k-i}
 If there is a~strictly positive function in $\Fk$,  then $\Fk$ contains $\Fi$. 
  \end{theorem}

  \begin{proof} 
    Let $q_0\in\Fk$, $q_0>0$, and $q\in \Fi$.  To prove  $q\in\Fk$ we may, by Lemma~\ref{simple}(1),  assume  that
 $q<\infty$. Let $\ve>0$ and  let  $g\in\E^+$ be integrable such that   
 $A_x:=\{G(x,\cdot)q>g\}$ satisfies
   \begin{equation}\label{ax}
  \int_{A_x} G(x,\cdot)q\,d\mu<\ve, \qquad x\in E.
\end{equation}
Since $\{Mq_0<q\}\downarrow \emptyset$ as $M\uparrow\infty$,
there is $M\in\nat$ such that, defining  $A:=\{Mq_0<q\}$, 
\begin{equation}\label{def-A}
  \int_A g\,d\mu<\ve.
\end{equation}
Let  $q':=1_{E\setminus A} q$ and $q'':=1_A q$. Then $q=q'+q''$, $q'\le Mq_0$. 
By Lemma \ref{simple}(2),  $q'\in\Fk$.
Moreover, $G(x,\cdot)q\le g$ on $E\setminus A_x$. So, by  (\ref{ax}) and (\ref{def-A}), 
\begin{equation*} 
  Gq''(x)=\int_A G(x,\cdot)q\,d\mu \le \int_{A_x} G(x,\cdot)q\, d\mu +
 \int_{A\setminus A_x} g\,d\mu   < 2\ve, \qquad x\in E.
\end{equation*} 
By Lemma \ref{simple}(3), the proof is finished. 
\end{proof}

\begin{theorem}\label{i-k}
  If there is a~strictly positive function in $\Fi$,  then $\Fi$ contains $\Fk$.
\end{theorem}

\begin{proof} Let $q_0\in\Fi$, $q_0>0$, $q\in\Fk$ and $\ve>0$. By Lemma \ref{simple},(1), 
  $ G(1_{\{q> Mq_0\}}q)\downarrow 0$ pointwise as~$M\uparrow\infty$, hence uniformly on $E$,
  by compactness of~$K_q$.
    So there exists $M\in\nat$ such that
        \begin{equation}\label{pmg}
A:=\{q>Mq_0\} \mbox{ \  satisfies \ }      G(1_Aq)\le \ve. 
    \end{equation} 
By assumption, there is an integrable $g\in \E^+$ such that, for every $x\in E$,
       \begin{equation}\label{q0}
 A_x:=\{G(x,\cdot)q_0\ge g\} \mbox{ \  satisfies \ }      \int_{A_x} G(x,\cdot)q_0\,d\mu \le \ve/M. 
    \end{equation}
    Let us fix $x\in E$ and define $B_x:=\{G(x,\cdot)q\ge Mg\}$. 
    Then
    \begin{equation*}
      B_x\setminus A\subset \{Mg\le G(x,\cdot)q\le  G(x,\cdot)Mq_0\}\subset A_x,
    \end{equation*}
    and hence,   by   (\ref{q0}), 
    \begin{equation*}
      \int_{B_x\setminus A}  G(x,\cdot)q \, d\mu\le M \int_{A_x}  G(x,\cdot)q_0\,d\mu \le \ve .
      \end{equation*} 
      Thus $\int_{B_x}  G(x,\cdot)q \, d\mu\le 2\ve$, by (\ref{pmg}). Since $Mg$ is  integrable,
      the proof is finished.
\end{proof}

 \begin{corollary}\label{equality}
If both $\Fi$ and $\Fk$ contain strictly positive   functions, then $ \Fk =\Fi$.
\end{corollary}

Let $G'\colon E\times E\to [0,\infty]$ be $\E\otimes \E$-measurable and suppose that $G'\le G$.
Then, of course,  $\Fi\subset \F\!_{ui}(G')$. So Theorems \ref{i-k} and \ref{k-i} also imply the following.

\begin{corollary}\label{}   If both $\Fi$ and $\F\!_{co}(G')$ contain strictly positive functions,
  then $\Fk\subset \F\!_{co}(G')$.
  \end{corollary} 

  \begin{remark}{\rm
      Of course, the preceding results immediately yield corresponding statements for arbitrary $\E$-measurable
      numerical  functions $q$ on $E$ using $q=q^+-q^-$.
    }
\end{remark}
 
 \section{Examples}\label{suff-ex}
 
 \subsection{First example}
 Let $E$ be a~Borel set in $\reald$, $d\ge 1$, let $\E$ be the $\sigma$-algebra of all Borel sets in $E$
 and~$\mu$ be the restriction of Lebesgue measure on $(E,\E)$. 
 For $x\in E$ and $r>0$, we define  $B(x,r):=\{y\in E\colon
 |y-x|<r\}$.
 Let $G\colon E\times E\to [0,\infty]$ be  measurable and 
 let $\vp\colon  [0,\infty)\to [0,\infty]$ be  measurable such that 
 \begin{equation*}
   G(x,y) \le \vp(|x-y|), \qquad x,y\in E,
 \end{equation*}
 and, for some $a, r,\in (0,\infty)$,  $\vp\le a$ on $(r,\infty)$  and $\int_0^r \vp(t) t^{d-1} dt<\infty$.
 
   \begin{proposition}\label{app}
        \begin{equation*} 
     \F:=\{q\in\E^+\colon
     q\mbox{ integrable, } q\le 1,\  \lim\nolimits_{|x|\to\infty} q(x)=0\}
     \end{equation*} 
     is contained in $\Fi$, and $\Fk\subset \Fi$.
   \end{proposition} 
 
   \begin{proof} By Theorem \ref{i-k}, it clearly suffices to prove $\F\subset \Fi$. 
To that end  we may assume without loss of generality that
     \begin{equation*}
       E=\reald \und G(x,y)=\vp(|x-y|), \qquad x,y\in \reald
       \end{equation*} 
       (first extend $G$ to $\reald\times \reald$ by $G(x,y):=0$, if~$x$ or~$y$ are
       in the complement of $E$).

     Let $q\in \F$.         Defining $ f_x:= 1_{B(x,r)} G(x,\cdot)$, we then have
     \begin{equation*}
          G(x,\cdot) q \le aq+  f_xq, \qquad x\in E.
        \end{equation*}
        So it suffices to show that the functions $f_xq$, $x\in E$, are uniformly integrable.

        Let  $\ve>0$ and  $b:=1+\int f_0\,d\mu$. Then $b<\infty$, by assumption on $\vp$,   and
        we may choose $R>0$ such that $q\le \ve/b$ on~$ B(0,R)^c$. 
     If $x\in B(0,R+r)^c$, then  $B(x,r)\cap B(0,R)=\emptyset$, and hence
     \begin{equation}\label{weit}
       \int f_xq\,d\mu\le \frac \ve b \int f_x\,d\mu=\frac \ve b\int f_0\,d\mu\le \ve.
     \end{equation}
   Suppose now that $x\in B(0,R+r)$, and hence $B(x,r)\subset B(0,R+2r)$.  
Let $M>a$ such that  $\int_{\{f_0>M\}} f_0\,d\mu<\ve$ and  $g:=M1_{B(0,R+2r)} $.
 Then
         \begin{equation}\label{nah}
      \int_{\{f_xq\ge g\}} f_xq\,d\mu\le \int_{\{f_x\ge M\}} f_x\,d\mu =\int_{\{f_0\ge M\}} f_0\,d\mu <\ve.
    \end{equation}
    So the functions $f_xq$, $x\in E$, are uniformly integrable.
   \end{proof}

  \subsection{Second example}

  Let $(X,\W)$ be a~balayage space such that  $\W$ contains a~function $0<w_0\le 1$
  (see \cite{BH, H-prag-87, H-course, H-equi-compact}),
  and let $G\colon X\times X\to [0,\infty]$
  be Borel measurable such that,  for every $y\in X$,
  $G(\cdot,y)$ is a~potential on $X$ which is harmonic on $X\setminus \{y\}$. 

  Let $\mu$ be a~positive Radon measure on $X$ and let  $\B(X)$, $\C(X)$, respectively, denote
  the set of all Borel measurable numerical functions, continuous real functions, respectively,  on~$X$. 
  We recall that, for every positive $f\in\B(X)$,    the function $Gf:=\int G(\cdot,y) f(y)\,d\mu(y)$ 
is  lower semicontinuous,   by Fatou's Lemma.

\begin{proposition}\label{bal-ex}
  If there exists $q\in\B(X)$, $q>0$, such that $Gq\in\C(X)$,
  then $\Fi\subset \Fk$.
  \end{proposition} 
  
\begin{proof}      By Theorem \ref{k-i}, it suffices to find  a~strictly positive function
    $q_0\in\Fk$.

    We choose compact sets $L_n$, $n\in\nat$, covering $X$. Let $n\in\nat$ and
    \begin{equation*}
      q_n:=1_{L_n}q. 
    \end{equation*}
    Since $Gq_n+G(1-q_n)=Gq\in \C(X)$, we know that
    $q_n$ is a~continuous real potential. It is  harmonic on $X\setminus L_n$.
    Let $a_n:= \sup Gq_n(L_n)/\inf w_0(L_n)$. Obviously, $Gq_n\le a_n w_0$ on $L_n$, hence on~$X$. 
  So $Gq_n\le a_n$ and, by  \cite[Proposition 4.1]{H-equi-compact}, $q_n\in\Fk$. 
    Using Lemma \ref{simple}, we finally obtain that
    \begin{equation*}
      q_0:=\sum\nolimits_{n\in\nat} (a_n2^n)\inv  q_n\in \Fk.
      \end{equation*} 
        \end{proof}

 \end{document}